\documentclass[11pt]{amsart}
\usepackage[margin=0.75in]{geometry}

\usepackage{amsthm}
\usepackage{amssymb}
\usepackage{amsmath}
\usepackage{comment}
\usepackage{thm-restate}
\usepackage{url} 
\usepackage{hyperref}
\usepackage[noabbrev,capitalise]{cleveref}

\usepackage[utf8]{inputenc}

\usepackage{color}
\usepackage[normalem]{ulem}

\newcommand{\mc}[1]{\mathcal{#1}}

\newcounter{i}
\theoremstyle{plain}
\newtheorem{thm}{Theorem}[section]
\newtheorem{lem}[thm]{Lemma}
\usepackage{color}

\usepackage[dvipsnames, table]{xcolor}

{\noindent \emph{Proof.} {}{#1}{}}{\hfill
	$\Diamond$\vspace{1em}}

\usepackage{amsthm,thmtools}

\usepackage{etoolbox}
 \AtEndEnvironment{proof}{\setcounter{claim}{0}}

\theoremstyle{plain} % just in case the style had changed
\newcommand{\thistheoremname}{}
\newtheorem{genericthm}{\thistheoremname}

\theoremstyle{definition}
\newtheorem{definition}[thm]{Definition}

\DeclareMathOperator{\on}{\mc{O}{\rm n}}
\DeclareMathOperator{\off}{\mc{O}{\rm ff}}
\DeclareMathOperator{\hleft}{\mc{L}}
\DeclareMathOperator{\hright}{\mc{R}}

\title{A Short Proof of the Existence of $K_q^r$-absorbers}

\author{
Michelle Delcourt
}
\address{Department of Mathematics, Toronto Metropolitan University (formerly named Ryerson University), Toronto, Ontario M5B 2K3, Canada}
\email{mdelcourt@torontomu.ca}
\thanks{Delcourt's research supported by NSERC under Discovery Grant No.\ 2019-04269.}

\author{
Tom Kelly
}
\address{
School of Mathematics, Georgia Institute of Technology, Atlanta, GA 30332, USA
}
\email{tom.kelly@gatech.edu}
\thanks{Kelly's research supported by the National Science Foundation under Grant No.\ DMS-2247078.}
\author{
Luke Postle
}
\address{Combinatorics and Optimization Department, University of Waterloo, Waterloo, Ontario N2L 3G1, Canada}
\email{lpostle@uwaterloo.ca}
\thanks{Postle's research partially supported by NSERC under Discovery Grant No.\ 2019-04304.}

\begin{document}

\maketitle

\begin{abstract} 
We codify a short self-contained proof of the existence of $K_q^r$-absorbers implicit in Keevash's~\cite{K14} original proof of the Existence Conjecture. Combining this with the work of the first and third authors in~\cite{DPI} yields a proof of the Existence Conjecture for Combinatorial Designs that is not reliant on the construction of $K_q^r$-absorbers by Glock, K\"uhn, Lo, and Osthus \cite{GKLO16}.
\end{abstract}

\section{Introduction}

The famous Existence Conjecture for Combinatorial Designs from the 1800s posits that $(n,q,r,\lambda)$-designs exist whenever the necessary divisibility conditions are satisfied provided $n$ is large enough. In 2014, Keevash~\cite{K14} proved the Existence Conjecture using \emph{randomized algebraic constructions}. Thereafter in 2016, Glock, K\"{u}hn, Lo, and Osthus~\cite{GKLO16} gave a purely combinatorial proof of the Existence Conjecture via \emph{iterative absorption}. In 2024, the first and third authors~\cite{DPI} gave a third shorter one-step combinatorial proof of the Existence Conjecture via \emph{refined absorption} {\bf assuming the existence of $K_q^r$-absorbers}. The existence of $K_q^2$-absorbers was first proved by Barber, K\"{u}hn, Lo, and Osthus~\cite{BKLO16}. For $r\ge 3$, the existence of $K_q^r$-absorbers was first proved by Glock, K\"{u}hn, Lo, and Osthus~\cite{GKLO16} and was essential for their proof of the Existence Conjecture. Indeed, their proof of the existence of $K_q^r$-absorbers uses induction on the uniformity $r$ and relies on the Existence Conjecture holding for smaller uniformities.  
Consequently, their absorbers also lack quantitative properties useful for variants and strengthenings of the Existence Conjecture which $K_q^2$-absorber constructions are known to have (see~\cite{GKLMO19} and~\cite{DKPIII}). Here in this paper we provide a short proof of the existence of $K_q^r$-absorbers as follows.  

If $F$ is an $r$-graph, an \emph{$F$-decomposition} of an $r$-graph $G$ is a partition of the edges of $G$ into copies of $F$. The complete $r$-graph on $q$ vertices, denoted $K_q^r$, is the $r$-graph with vertex set $[q]$ and edge set $\binom{[q]}{r}$. An $(n,q,r)$-Steiner system is equivalent to a $K_q^r$-decomposition of $K_n^r$. A necessary condition for an $r$-graph $G$ to admit a $K_q^r$-decomposition is that $G$ is \emph{$K_q^r$-divisible}, that is, $\binom{q-i}{r-i}~|~|\{e\in G: S\subseteq e\}|$ for all $0\le i \le r-1$ and $S\subseteq V(G)$ with $|S|=i$.

\begin{definition}[Absorber]
Let $L$ be a $K_q^r$-divisible $r$-graph. An $r$-graph $A$ is a \emph{$K_q^r$-absorber} for $L$ if $V(L)\subseteq V(A)$ is independent in $A$ and both $A$ and $L\cup A$ admit $K_q^r$-decompositions.
\end{definition}

We note this is the definition of absorber used by Barber, K\"uhn, Lo, Osthus~\cite{BKLO16} and also in the paper ``Minimalist Designs" by Barber, Glock, K\"uhn, Lo, Montgomery, and Osthus~\cite{BGKLMO} which differs from the definition of absorber given in Glock, K\"uhn, Lo, and Osthus~\cite{GKLO16} and the survey paper of Glock, K\"uhn, and Osthus~\cite{GKO}. Those papers only require $A$ to be edge-disjoint from $L$ but it is natural to require the stronger property that $V(L)$ is independent in $A$ since this is useful for embedding purposes (as done in~\cite{GKLO16, DPI}). Nonetheless, the existence of $K_q^r$-absorbers with the stronger independence definition also follows from the main absorbing lemma of Glock, K\"uhn, Lo, and Osthus~\cite[Lemma 8.2]{GKLO16}.

\begin{thm}[Glock, K\"uhn, Lo, and Osthus~\cite{GKLO16}]\label{thm:AbsorberExistence}
Let $q > r\ge 1$ be integers. If $L$ is a $K_q^r$-divisible $r$-graph, then there exists a $K_q^r$-absorber for $L$.
\end{thm}

Here we extract a short proof of Theorem~\ref{thm:AbsorberExistence} that was implicit in Keevash's original proof of the Existence Conjecture~\cite{K14}. That said, our proof differs from Keevash's original in that it sidesteps the use of randomized algebraic constructions by using a layering technique for constructing an \emph{orthogonal booster} from a \emph{booster} (see Definition~\ref{def:Booster} below) that we previously used for constructing ``spread boosters" and ``girth boosters" in~\cite{DKPIV,DPII}. The layered construction of orthogonal boosters has better degeneracy and density properties than the randomized algebraic one. We should mention that Keevash in his most recent proof of the Existence Conjecture~\cite{K24} independently adopted this layering technique for orthogonal booster construction (see the proof of~\cite[Lemma 3.1]{K24}) in his construction of an omni-absorber, presumably for its better numeric properties but also to avoid the use of randomized algebraic constructions. Nevertheless, we think codifying this short proof on its own is useful for the community but additionally will prove useful for future applications on variants of the Existence Conjecture.

\section{Proof}\label{s:Proof}

In the proof, we use the convention of identifying a hypergraph with its edge set; that is, for an $r$-graph $G$, we use $G$ to denote its edge set and $V(G)$ to denote its vertex set.  
We also call an $r$-graph $S$ which is isomorphic to $K^r_q$ a \textit{$K^r_q$-clique}.
For a $K_q^r$-decomposition $\mc{Q}$ of an $r$-graph $G$ and an edge $e\in G$, we use $\mc{Q}[e]$ to denote the (unique) $K_q^r$-clique in $\mc{Q}$ containing $e$.

\begin{definition}[Booster]\label{def:Booster}
A \emph{$K_q^r$-booster} for a $K_q^r$-clique $S$ is an $r$-graph $B$ such that $V(B)\subseteq V(S)$ is independent in $S$, $B$ has a $K_q^r$-decomposition $\off (B)$, and $B\cup S$ has a $K_q^r$-decomposition $\on (B)$ such that $S\not\in \on(B)$. 

Furthermore, we say $B$ is \emph{orthogonal} if for all distinct $e, f\in S$, we have $\on(B)[e]\ne \on(B)[f]$.
\end{definition}

\begin{definition}[Hinge]
Let $S$ and $S'$ be two $K_q^r$-cliques such that $S\cap S'=\{e\}$ for some $r$-edge $e$. A \emph{$K_q^r$-hinge} for $S$ and $S'$ is an $r$-graph $H$ such that $H$ is edge-disjoint from $S\cup S'$, $H\cup (S\setminus e)$ has a $K_q^r$-decomposition $\hleft(H)$ (the ``left" one), and $H\cup (S'\setminus e)$ has a $K_q^r$-decomposition $\hright(H)$ (the ``right" one). 

Furthermore, we say $H$ is \emph{independent} if $V(S)\cup V(S')$ is independent in $H$.
\end{definition}

We let $K_{q*n}^r$ denote the complete $q$-partite $r$-graph with $n$ vertices in each part (that is with vertex set $[q]\times [n]$ and edge-set consisting of all $r$-sets with at most one vertex in each part).

\begin{lem}\label{lem:Booster}
Let $q > r\ge 1$ be integers. If $S$ is a $K^r_q$-clique, then there exists a $K_q^r$-booster for $S$.
\end{lem}
\begin{proof}
Let $n$ be a prime such that $2q-r < n < 2(2q-r)$ (such exists by Bertrand's postulate, first proved in 1852 by Chebyshev~\cite{C52}). We may assume $S$ is a copy of $K_q^r$ in $K_{q*n}^r$. We prove that $K_{q*n}^r\setminus S$ is a $K_q^r$-booster for $S$. To that end, it suffices to show that $K_{q*n}^r$ admits two $K_q^r$-decompositions $\mc{Q}_1,\mc{Q}_2$ that do not share a clique as follows since then if $S\in \mc{Q}_1$, we have that $\mc{Q}_1\setminus S$ and $\mc{Q}_2$ are as desired. 

Let $M$ be a $(q-r)\times q$ Cauchy matrix over $\mathbb{F}_n$ (specifically, we can take $x_i := i$ for $i\in [q-r]$, $y_j := q-r+j$ for $j\in [q]$ and $M_{ij} := 1/(x_i-y_j)$ for $i\in [q-r], j\in [q]$). By construction, every submatrix of a Cauchy matrix is also a Cauchy matrix; a result of Cauchy~\cite{C1841} implies that every square Cauchy matrix is invertible. A vector $v\in \mathbb{F}_n^q$ naturally corresponds to a copy of $K_q^r$ in $K_{q*n}^r$. For $a \in \mathbb{F}_n^{q-r}$, the set of solutions $v\in \mathbb{F}_n^q$ of $Mv=a$ naturally corresponds to a $K_q^r$-decomposition $\mc{Q}_a$ of $K_{q*n}^r$ (since every $r$-set will be in exactly one such $K_q^r$, or equivalently exactly one such solution vector $v$, as every $(q-r)\times (q-r)$ submatrix of $M$ is invertible). Indeed, the family of decompositions $(\mc{Q}_a: a \in \mathbb{F}_n^{q-r})$ partition the $K_q^r$'s in $K_{q*n}^r$ into $K_q^r$-decompositions of $K_{q*n}^r$ (since for every vector $v$, there is exactly one $a$ such that $Mv=a$). Since there are at least two choices $a_1,a_2$ of $a$ as $n > 1$, there exist two decompositions $\mc{Q}_{a_1}, \mc{Q}_{a_2}$ of $K_{q*n}^r$ that do not share a clique.
\end{proof}

Now we can layer boosters (using our technique from~\cite{DKPIV,DPII}) to build an orthogonal booster as follows.

\begin{lem}\label{lem:OrthBooster}
Let $q > r\ge 1$ be integers. If $S$ is a $K^r_q$-clique, then there exists an orthogonal $K_q^r$-booster for $S$.   
\end{lem}
\begin{proof}
Let $B$ be a $K_q^r$-booster for $S$ such that $i$, the number of cliques of $\on(B)$ that contain at least one edge of $S$, is maximized. Note such exists by Lemma~\ref{lem:Booster}; furthermore we have that $i\ge 2$ by definition of booster. If $i=\binom{q}{r}$, then $B$ is orthogonal as desired. 

So we assume $i < \binom{q}{r}$. Hence there exists a clique $Q\in \on(B)$ such that $Q$ contains at least two edges of $S$. Let $j:= |V(Q)\cap V(S)|$. Note $j\ge r+1$ since $Q$ contains at least two edges of $S$. Since $Q\ne S$, it follows that $j < q$. Let $B^*$ be a $K_q^r$ booster for $Q$ as given by Lemma~\ref{lem:Booster}.  By relabelling vertices, we assume without loss of generality that $V(B^*) \cap V(B) = V(Q)$ and in addition that $V(Q) \cap V(S)$ is not contained in a clique of $\on(B^*)$, as we now proceed to show is possible.

Since $\on(B^*)$ has at least two cliques that contain at least one edge of $Q$, there exists some subset $T_0\subseteq V(Q)$ with $|T_0|=r+1$ such that $T_0$ is not contained in a clique of $\on(B^*)$. Let $T\subseteq V(Q)$ with $T_0\subseteq T$ and $|T|=j$. Note then that $T$ is not contained in a clique of $\on(B^*)$ and hence letting $V(Q)\cap V(S)=T$ is as desired.

Let $B'=B\cup B^*$. Let $\on(B') := (\on(B)\setminus \{Q\})\cup\on(B^*)$ and $\off(B'):= \off(B)\cup \off(B^*)$. Note that $\on(B')$ is a $K_q^r$-decomposition of $B'\cup S$ and $\off(B')$ is a $K_q^r$-decomposition of $B'$. Moreover, the number of cliques of $\on(B')$ intersecting $S$ is at least $i+1$, and hence $B'$ contradicts the maximality of $B$.
\end{proof}

We note that if $B$ is an orthogonal $K_q^r$-booster for $S$ and $e\in S$ and we let $S'=\on(B)[e]$, then $H:=B\setminus (S'\setminus e)$ is a $K_q^r$-hinge for $S$ and $S'$ where $\hleft(H):= \on(B)\setminus \{S'\}$ and $\hright(H):= \off(B)$. 

\begin{lem}\label{lem:IndHinge}
Let $q > r\ge 1$ be integers. If $S_1$ and $S_2$ are $K^r_q$-cliques such that $S_1 \cap S_2 = \{e\}$ for some $r$-edge $e$, then there exists an independent $K_q^r$-hinge for $S_1$ and $S_2$.    
\end{lem}
\begin{proof} 
Let $S$ be a $K^r_q$-clique where $S_1 \cap S = S_2 \cap S = \{e\}$.
As mentioned above it follows from Lemma~\ref{lem:OrthBooster} that there exists a $K_q^r$-hinge $H_1$ for $S_1$ and $S$ and a $K_q^r$-hinge $H_2$ for $S$ and $S_2$.  By relabelling vertices, we assume without loss of generality that $V(H_1) \cap V(H_2) = V(S)$.
Let $H := H_1\cup (S\setminus \{e\}) \cup H_2$.  Let $\hleft(H) := \hleft(H_1)\cup \hleft(H_2)$ and $\hright(H) := \hright(H_1)\cup \hright(H_2)$. Then $H$ is an independent $K_q^r$-hinge for $S_1$ and $S_2$ as desired. 
\end{proof}

We need the following classical theorem on the existence of integral $K_q^r$-decompositions proved by Graver and Jurkat~\cite{GJ73} and independently Wilson~\cite{W73} in 1973 (whose shortest proof is at most a page say, for example see Section 7 of Keevash~\cite{K24}).  An \emph{integral $F$-decomposition} of an $r$-graph $G$ is an assignment of integers $w_Q$ to copies $Q$ of $F$ in $K_{V(G)}^r$ such that for every $e\in G$, $\sum_{Q\ni e} w_Q = +1$, and for every $e\in K_{V(G)}^r\setminus G$, $\sum_{Q\ni e} w_Q = 0$.

\begin{thm}[Graver and Jurkat~\cite{GJ73}, Wilson~\cite{W73}]\label{thm:Integral}
Let $q > r\ge 1$ be integers. If $L$ is a $K_q^r$-divisible $r$-graph with $|V(L)|\ge q+r$, then there exists an integral $K_q^r$-decomposition of $L$.
\end{thm}

We are now prepared to prove Theorem~\ref{thm:AbsorberExistence} as follows.

\begin{proof}[Proof of Theorem~\ref{thm:AbsorberExistence}]
We assume without loss of generality that $|V(L)|\ge q+r$ since it suffices to prove the theorem for this case (to see this, if $|V(L)| < q+r$, we let $L'$ be obtained from $L$ by adding isolated vertices such that $|V(L')|=q+r$; then we find a $K_q^r$-absorber $A$ for $L'$ which is then also a $K_q^r$-absorber for $L$). By Theorem~\ref{thm:Integral}, there exists an integral $K_q^r$-decomposition $\Phi$ of $L$ (with integers $w_Q$ for each $Q\in \binom {V(L)}{q}$). We view $\Phi$ as a multi-set of positive cliques $\Phi^+$ and a multi-set of negative cliques $\Phi^-$ (where each clique $Q$ appears with multiplicity $|w_Q|$). 

For each edge $e\in L$, let $M_e$ be a directed matching from the elements of $\Phi^-$ containing $e$ to all but one of the elements of $\Phi^+$ containing $e$. For each $r$-set $f$ in $\binom{V(L)}{r}\setminus L$, let $M_f$ be a directed matching of the elements from $\Phi^-$ containing $f$ to the elements of $\Phi^+$ containing $f$.  

We construct a $K_q^r$-absorber $A$ for $L$ as follows:
\begin{equation*}
    A := \left.\bigcup_{S \in \Phi^+ \cup \Phi^-}B_S\right. \cup \left.\bigcup_{f \in \binom{V(L)}{r},Q_1Q_2\in M_f}H_{f,Q_1,Q_2}\right.\text{\qquad where}
\end{equation*}
\begin{itemize}\itemsep.05in
    \item for each clique $S \in \Phi^+ \cup \Phi^-$, $B_S$ is an orthogonal $K_q^r$-booster for $S$ whose existence is guaranteed by Lemma~\ref{lem:OrthBooster}, and
    \item for each $f \in \binom{V(L)}{r}$ and each $Q_1Q_2 \in M_f$,  $H_{f,Q_1,Q_2}$ is an independent $K^r_q$-hinge for $\on(B_{Q_1})[f]$ and $\on(B_{Q_2})[f]$ whose existence is guaranteed by Lemma~\ref{lem:IndHinge} since 
     $\on(B_{Q_1})[f] \cap \on(B_{Q_2})[f] = \{f\}$ as $B_{Q_1}$ and $B_{Q_2}$ are orthogonal boosters. 
\end{itemize}

We assume without loss of generality that $V(B_S)\setminus V(S)$ is disjoint from every other booster. Similarly we assume without loss of generality that $V(H_{f,Q_1,Q_2})\setminus (\on(B_{Q_1})[f] \cup \on(B_{Q_2})[f])$ is disjoint from every other booster and hinge.
    
Now we proceed to verify that $A$ is a $K_q^r$-absorber for $L$ as follows. First we note that $V(L)$ is independent in $A$ and $A$ is simple (that is the boosters and hinges are all pairwise edge-disjoint). This follows from the definition of booster (namely that $V(S)$ is independent in $B$) and the fact that each hinge in the construction is an independent hinge. 

It remains to show that there exist $K_q^r$-decompositions $\mc{A}_1$ of $A$ and $\mc{A}_2$ of $A\cup L$ as follows:
\begin{align*} &\mc{A}_1: = \bigcup_{S\in \Phi^-} \bigg( \mc{O}{\rm n}(B_S)~\Big\backslash~~~~~\bigcup_{e\in S}~~~~~~\mc{O}{\rm n}(B_S)[e] \bigg) ~&\cup~&\bigcup_{f\in \binom{V(L)}{r}, Q_1Q_2\in M_f} \hleft(H_{f,Q_1,Q_2})~&\cup~&\bigcup_{S\in \Phi^+} \mc{O}{\rm ff}(B_S),\\
&\mc{A}_2 := \bigcup_{S\in \Phi^+} \bigg(\on(B_S)~\Big\backslash~\bigcup_{e\in S: S\in V(M_e)}\on(B_S)[e] \bigg)~&\cup~&\bigcup_{f\in \binom{V(L)}{r}, Q_1Q_2\in M_f} \hright(H_{f,Q_1,Q_2})~&\cup~&\bigcup_{S\in \Phi^-} \mc{O}{\rm ff}(B_S).\end{align*}
(Or in words: for $\mc{A}_1$, we use the \emph{on} decompositions of the negative clique boosters except we do not use the negative orthogonal cliques themselves; we use the \emph{left} decomposition of the hinges to decompose the edges of those unused cliques; finally we use the \emph{off} decompositions of the positive clique boosters.

For $\mc{A}_2$, we use the \emph{on} decompositions of the positive clique boosters except we do not use the positive orthogonal cliques themselves (except for those unmatched ones which will decompose the edges of $L$); we use the \emph{right} decomposition of the hinges to decompose the edges of those unused cliques; finally we use the \emph{off} decompositions of the negative clique boosters.)
\end{proof}

\bibliographystyle{plain}
\bibliography{ref2}

\end{document}